\documentclass[11pt]{article}
\usepackage{amssymb,amsmath,amsthm,mathdesign}
\usepackage{pdfsync}
\usepackage{color}
\usepackage[colorlinks]{hyperref}

\newtheorem{theorem}{Theorem}[section]
\newtheorem{lemma}[theorem]{Lemma}
\newtheorem{proposition}[theorem]{Proposition}
\newtheorem{corollary}[theorem]{Corollary}
\newtheorem{assumption}[theorem]{Assumption}
\newtheorem{remark}[theorem]{Remark}
\newtheorem{example}[theorem]{Example}
\numberwithin{equation}{section}

\newcommand{\be}{\begin{equation}}
\newcommand{\ee}{\end{equation}}

\newcommand{\bes}{\begin{equation*}}
\newcommand{\ees}{\end{equation*}}

\def\E{\bE}

\def\cG{\mathcal{G}}

\def\bE{\mathbb{E}}

\newcommand{\R}{\mathbf{R}}

\newcommand{\1}{\boldsymbol{1}}
\renewcommand{\d}{{\rm d}}

\renewcommand{\geq}{\geqslant}
\renewcommand{\leq}{\leqslant}

\def\m1{\mathbf{1}}

\allowdisplaybreaks[1]

\allowdisplaybreaks[1]

\title{Some non-existence results for a class of stochastic partial differential equations.}
\author{Mohammud Foondun\\University of Strathclyde
\and Wei Liu \\ Shanghai Normal University
\and Erkan Nane\\ Auburn University }
\date{}
\begin{document}
\maketitle

\begin{abstract}
Consider the following stochastic partial differential equation,
\begin{equation*}
\partial_t u_t(x)=
 \mathcal{L}u_t(x)+ \sigma (u_t(x))\dot F(t,x)\quad{t>0}\quad\text{and}\quad x\in \R^d.
\end{equation*}
The operator $\mathcal{L}$ is the generator of a strictly stable process and $\dot F$ is the random forcing term which is assumed to be Gaussian. Under some additional conditions, most notably on $\sigma$ and the initial condition, we show non-existence of global random field solutions. Our results are new and complement those of P. Chow in \cite{chow2} and  \cite{chow1}.
\end{abstract}
\section{Introduction and main results}
Consider
\begin{equation*}
\partial_t u_t(x)=
 \mathcal{L}u_t(x)+ \sigma (u_t(x)) \dot F(t,x)\quad \text{for all}\quad t>0 \quad\text{and}\quad x\in \R^d.
\end{equation*}
Here $\mathcal{L}$ denotes the fractional Laplacian, the generator of an $\alpha$-stable process and $\dot F$ is the random forcing term which we will take to be white in time and possibly colored in space. The initial condition will always be assumed to be a non-negative bounded deterministic function.  The function $\sigma$ is a locally Lipschitz function. The main aim of this paper is to show that under some additional conditions on the initial condition and the function $\sigma$, equations of the above type cannot have global random field solutions. By `global' existence, we mean that the solution exists for all times.  The failure of global solutions usually manifests itself via the `blow up' of certain quantities involving the solution. In this paper, we will be concerned with the moments of the solution.

Blow-up or non-existence problems are of interest from a theoretical point of view. They are also very useful for applied researchers. Physically, blow-up might represent cracks and various other singularities; see \cite{Levine} and \cite{DL} for more information regarding these questions in the deterministic settings. Even though for stochastic equations, the literature regarding these types of problems is not that rich, blow-up of stochastic partial differential equations have started to receive some attention. Mueller and Sowers in \cite{Mueller2000,MuellerSowers1993} points out that the time-space white noise driven stochastic heat equations with Dirichlet boundary condition will blow up in finite time with positive probability, if $\sigma(u) = u^{\gamma}$ with $\gamma >3/2$. When a drift function $f$ is taken into consideration, Bonder and Groisman in \cite{FernandezGroisman2009} discuss equations of the following type
\begin{equation}
\label{generalSPDEwithdrift}
\partial_t u_t(x)=
 \mathcal{L}u_t(x) +  f(u_t(x))+ \sigma (u_t(x)) \dot F(t,x)\quad \text{for all}\quad t>0 \quad\text{and}\quad x\in D \subset \R^d,
\end{equation}  
with  $\sigma(u)$ to be a constant, $\dot F(t,x)$ to be the space-time white noise and $D = (0,1)$.
The authors reveal that the solution will blow up in finite time with probability one for every initial data, if the drift function is nonnegative convex and admits $\int^{\infty} 1/f < \infty$. Lv and Duan in \cite{LvDuan2015} investigate (\ref{generalSPDEwithdrift}) in higher spatial dimension and reveal the effects of the interplay between $f$ and $\sigma$ on the finite time blow-up of the solution in the moment sense. Bao and Yuan in \cite{BaoYuan2016} study the finite time blow-up in $L^p$-norm of stochastic reaction-diffusion equations with jumps within a bounded domain. Li, Peng and Jia in \cite{LiPengJia2016} consider the blow-up in $L^p$-norm for a class of L\'evy noise driven SPDEs, which extends the results in \cite{LvDuan2015}.

Our work complements a series of very interesting results by P. Chow who in \cite{chow1} shows non-existence of global solutions in the $L^p$-norm. As opposed to all the papers mentioned above, here our equations are defined on the whole space. Another difference is that our equations involve a non-local operator instead of a local one as in all the papers mentioned above. We will describe our method with more precision later on. But it is worth noting that here we employ completely different techniques relying mostly on sharp heat kernel estimates. We have worked quite hard to present our method in a simple way because we believe that it is versatile enough to be applied to various other equations. In fact, our last theorem is much closer in spirit to the results in \cite{chow2} and \cite{chow1} in that we look at equations on bounded domain; see Theorem \ref{Dirichlet} and the discussion preceding it.  Our results can also be seen as a continuation of the work of Foondun and Parshad \cite{Foondun-Parshad-15} on the non-existence of finite energy solutions of parabolic stochastic partial differential equations.

We now proceed to describe our results with more precision. We will work with white and colored noise driven equations and present the results pertaining to each type of equations separately. Firstly, we will look at the following
\begin{equation}\label{eq:white}
\partial_t u_t(x)=
 \mathcal{L}u_t(x)+ \sigma (u_t(x)) \dot W(t,x)\quad{t>0}\quad\text{and}\quad x\in \R^d.
\end{equation}

A mild solution to \eqref{eq:white} in the sense of Walsh \cite{walsh} is any $u$ which is adapted to the filtration generated by the white noise and satisfies the following evolution equation
\begin{equation}\label{mild}
u_t(x)=
(\cG u)_t(x)+ \int_{\R^d}\int_0^t p_{t-s}(x-y)\sigma(u_s(y))W(\d s\,\d y),
\end{equation}
where
\begin{equation}\label{deter}
(\cG u)_t(x):=\int_{\R^d} p_t(x-y)u_0(y)\,\d y,
\end{equation}
and $p_t(\cdot)$ denotes the heat kernel of the fractional Laplacian of order $\alpha$. If we further have
\begin{equation}\label{moments}
\sup_{x\in[0,\,L]}\sup_{t\in[0,\,T]} \E|u_t(x)|^k<\infty \quad\text{for all}\quad T>0 \quad\text{and}\quad k\in[2,\,\infty],
\end{equation}
then we say that $u$ is a {\it random field} solution. Usually, existence is proved under the assumption that $\sigma$ is globally Lipschitz. But this can be proved under the local Lipschitz condition as well. To see this, define 
\begin{equation*}
\tau_N:= \inf \{t > 0, \sup_{x\in \R^d}|u_t(x)| > N\},
\end{equation*} 
then clearly $\tau_N$ is a stopping time representing the first time that $\sup_{x\in \R^d}|u_t(x)|$ exits $N$. We now have $|\sigma(u_s(x))-\sigma(u_s(y))|\leq K_N |u_s(x)-u_s(y)|$ for any $s \leq \min(t,\tau)$, where $K_N$ is a constant dependent on $N$. Following the techniques in \cite{Davar} and \cite{walsh}, we can prove existence and uniqueness of a local solution in $(0, \min(t,\tau))$ provided that $1<\alpha<2$ and $d=1$; two conditions which we will be in force whenever we are dealing with the above equation. Throughout this paper, the initial condition $u_0$ will always be a non-negative bounded deterministic function.  We will also look at equations driven by colored noise. Consider
\begin{equation}\label{eq:colored}
\partial_t u_t(x)=
 \mathcal{L}u_t(x)+ \sigma (u_t(x)) \dot F(t,x)\quad{t>0}\quad\text{and}\quad x\in \R^d.
\end{equation}
The corresponding mild solution in the sense of Walsh \cite{walsh} is given by
\begin{equation}\label{mild}
u_t(x)=
(\cG u)_t(x)+ \int_{\R^d}\int_0^t p_{t-s}(x-y)\sigma(u_s(y))F(\d s\,\d y).
\end{equation}
We will again be interested in the random field solution. But for this equation, we will need to impose some conditions on the noise term. We have
\begin{align*}
\E[\dot F(s,x)\dot F(t,y)]=\delta_0(t-s)f(x,\,y),
\end{align*}
where $f(x,\,y)\leq \tilde{f}(x-y)$ and $\tilde{f}$ is a locally integrable function on $\R^d$ with a possible singularity at $0$. We will further assume that for any $\epsilon>0$,
\begin{align*}
\int_{|x|\leq \epsilon}\tilde{f}(x)\log\frac{1}{|x|}\,\d x<\infty\quad\text{if}\quad d=2
\end{align*}
or
\begin{align*}
\int_{|x|\leq \epsilon}\tilde{f}(x)\frac{1}{|x|^{d-3}}\,\d x<\infty\quad\text{if}\quad d\geq 3.
\end{align*}
As stated above, $\sigma$ is a locally Lipschitz function. If instead, we assume that $\sigma$ is globally Lipschitz, then the equations described above have unique global random field solutions, meaning that \eqref{moments} holds for {\it all} $T>0$; see \cite{Davar} for more information. What we have in mind here are functions $\sigma$ which are polynomial-like so that \eqref{moments} will be shown to hold for only some finite range of time.

As mentioned above, not a lot of work has been done in these types of questions for stochastic partial differential equations. But the situation is completely different for deterministic differential equations. Fujita in \cite{Fujita}, considered
\begin{equation*}
\partial_t u_t(x)=\Delta u_t(x)+u(x)^{1+\lambda}\quad\text{for}\quad t>0\quad x\in \R^d,
\end{equation*}
with initial condition $u_0$ and $\lambda>0$. Set $p_c:= 2/d$, Fujita showed that for $\lambda<p_c$, the only global solution is the trivial one. While for $\lambda>p_c$, global solutions exist whenever $u_0$ is small enough. See \cite{Fujita} for details. At first, this result might seem counterintuitive, but the right intuition is that for large $\lambda$, if the initial condition is small, then $u^{1+\lambda}$ is even smaller and the dissipative effect of the Laplacian prevents the solution to grow too big for blow-up to happen. And when $\lambda$ is close to zero, then irrespective of the size of the initial condition, the dissipative effect of the Laplacian cannot prevent blow up of the solution. In our setting, we are looking at the moments of the solution. There is still an interplay between the dissipative effect of the operator and the forcing term and we are able to shed light only part of the true picture.  We show that if the initial condition is large enough then there there is no global solution. It might very well be just like for the deterministic case, if the non-linearity is high enough, then for small initial condition, there exists global solutions.
 \begin{assumption}
The function $\sigma$ is a locally Lipschitz satisfying the following growth condition. There exist a $\gamma>0$ such that
\begin{equation}\label{growth}
\sigma(x)\geq |x|^{1+\gamma}\quad\text{for all}\quad x\in \R^d.
\end{equation}
\end{assumption}
We are now ready to describe our findings in detail. For our first set of results, we will assume that the initial condition is bounded below by a positive constant. We therefore set
\begin{equation}\label{minumum-values-u0}
\inf_{x\in\R^d}u_0(x):=\kappa.
\end{equation}

\begin{theorem}\label{theo1}
Let $u_t$ be be the solution to \eqref{eq:white} and suppose that $\kappa>0$. Then there exists a $t_0>0$ such that for all $x\in \R$,
\begin{equation*}
\E|u_t(x)|^2=\infty\quad\text{whenever}\quad t\geq t_0.
\end{equation*}
\end{theorem}
The above result states that provided that the initial function is bounded below, the second moment will eventually cease to be finite for white noise driven equations.  For equations with colored noise, we have a slightly more complicated picture.  We will need the following non degeneracy condition on the spatial correlation of the noise.
\begin{assumption}\label{color}
Fix $R>0$, then there exists some positive number $K_{f}$ such that
\begin{equation*}
\inf_{x,\,y\in B(0,\,R)}f(x,\,y)\geq K_{f}.
\end{equation*}
\end{assumption}
Of course the constant $K_f$ is allowed to depend on $R$ but since we will mostly set $R=1$ when using this condition, we do not specify the dependence on $R$. The following examples show that the assumption is also very mild.
\begin{example}
For the following list of examples  Assumption \ref{color} is satisfied:
\begin{itemize}
\item Riesz kernel: \begin{equation*}f(x,\,y)=\frac{1}{|x-y|^\beta} \quad\text{with} \quad\beta< d \wedge \alpha.\end{equation*}
\item The Exponential-type kernel: $f(x,\,y)=\exp[-(x\cdot y)].$
\item The Ornstein-Uhlenbeck-type kernels: $f(x,\,y)=\exp[-|x-y|^\alpha]$ with $\alpha\in(0,\,2].$
\item Poisson Kernels: $$f(x,\,y)=\left(\frac{1}{|x-y|^2+1}\right)^{(d+1)/2}.$$
\item Cauchy Kernels: $$f(x,\,y)=\sum_{j=1}^d\left(\frac{1}{1+(x_j-y_j)^2}\right).$$
\end{itemize}
\end{example}

\begin{theorem}\label{theo2}
Let $u_t$ be the solution to \eqref{eq:colored} and suppose that Assumption \ref{color} holds. Fix $t_0>0$, then there exists a positive number $\kappa_0$ such that for all $\kappa\geq \kappa_0$, and  $ x \in \R^d$ we have
\begin{equation*}
\E|u_t(x)|^2=\infty, \quad \text{whenever}\quad t\geq t_0.
\end{equation*}
\end{theorem}

Unlike in Theorem \ref{theo1}, to establish non-existence of the second moment, we require that the initial condition is large enough. We believe that such a condition might be required.  This is because we know  that the fact the noise is spatially correlated induces some extra dissipation effect. In fact, even in the case of the corresponding linear equation($\sigma(u)\propto u$), it is known that for some correlation functions the moments might not grow exponentially fast. See for instance \cite{ChenKim} and \cite{HuLeNualart}. However, if we focus our attention on the case when the correlation function is given by the Riesz Kernel, we have the following stronger result concerning the solution to \eqref{eq:colored}.
\begin{theorem}\label{Theo3}
Suppose that the correlation function $f$ is given by
\begin{align*}
f(x,\,y)=\frac{1}{|x-y|^\beta}\quad\text{with}\quad \beta<\alpha\wedge d.
\end{align*}
Then for $\kappa>0$, there exists a positive number $\tilde{t}$ such that for all $t\geq \tilde{t}$ and $x\in \R^d$,
\begin{equation*}
\E|u_t(x)|^2=\infty.
\end{equation*}
\end{theorem}
It is worth noting that when $\beta\rightarrow 1$, the corresponding colored noise converges to the white noise. So the above the theorem is `consistent' with Theorem \ref{theo1}. So far our results were given under the assumption that the initial condition is bounded below by a constant. We now proceed to remove this assumption.
\begin{assumption}\label{initial}
Suppose that initial condition is non-negative and satisfies the following,
\begin{equation*}
\int_{B(0,\,1)}u_0(x)\,\d x:=K_{u_0}>0.
\end{equation*}
\end{assumption}
We have taken $B(0,\,1)$ as a matter of convenience. In the result below, $u_t$ denotes the solution to the white noise driven equation.
\begin{theorem}\label{energy-white}
Let $u_t$ be the solution to \eqref{eq:white}. Then, under Assumption \ref{initial}, there exists a $t_0\geq 0$ such that for all $t\geq t_0$ and $x\in \R$,
\begin{equation*}
\E|u_t(x)|^2=\infty\quad\text{whenever}\quad K_{u_0}\geq K,
\end{equation*}
where $K$ is some positive constant.
\end{theorem}
We have a similar result for the coloured noise driven equation.
\begin{theorem}\label{energy-colored}
Let $u_t$ be the solution to \eqref{eq:colored}. Then, under Assumptions \ref{color} and \ref{initial}, there exists a $t_0\geq 0$ such that for all $t\geq t_0$ and $x\in \R^d$,
\begin{equation*}
\E|u_t(x)|^2=\infty\quad\text{whenever}\quad K_{u_0}\geq K,
\end{equation*}
where $K$ is a positive constant.
\end{theorem}
We mention that the constant $K$ appearing in the above two results might be different. The general strategy of  our method consists of obtaining non-linear renewal-type inequalities whose solutions blow up in finite time. Coming up with those inequalities can be difficult and this is where our method is novel. For the colored-noise case, a crucial observation is that one should look at the following quantity $\E|u_t(x)u_t(y)|$ instead of $\E|u_t(x)|^2$. We also need to have good control over the deterministic term $(\mathcal{G}u)_t(x)$. We make use of heat kernel estimates for short times and use the fact that we can restart the solution at a later time. All this will be made more clear in the proofs. As mentioned above, our method is soft and can be adapted to study a wider class of equations. For instance, our operator $\mathcal{L}$ can be more general. All we require is sufficient well behaved heat kernel estimates. In principle, we could also look at equations driven by discontinuous noises. 

Our final theorem, in some particular cases, extends those of \cite{chow2} and \cite{chow1}. We are going to consider the above equations in a ball  with Dirichlet boundary conditions. We will need the following slightly stronger condition on the initial condition. Fix $R>0$.  The ball $B(0,\,R)$ is going to be our domain. We will need  the following assumption.
\begin{assumption}\label{initial-dirichlet}
We assume that the initial condition $u_0$ is a non-negative function whose support, denoted by $S_{u_0}$ satisfies $B(0,\,R/2)\subset S_{u_0}$ such that $\inf_{x\in B(0,\,R/2)}u_0(x)>\tilde{\kappa}$ for some positive $\tilde{\kappa}$.
\end{assumption}

\begin{theorem}\label{Dirichlet}
Fix $R>0$ and consider
\begin{equation}\label{eq:dir}
\partial_t u_t(x)=
 \mathcal{L}u_t(x)+ \sigma (u_t(x))\dot F(t,x)\quad{t>0}\quad\text{and}\quad x\in B(0,\,R).
 \end{equation}
Here $\mathcal{L}$ is the generator of a stable process killed upon exiting the ball $B(0,\,R)$. The noise $\dot F$, when not space-time white noise is taken to be spatially colored with correlation function satisfying all the conditions stated above. Fix $\epsilon>0$, then there exist $t_0>0$ and $K>0$, such that for $K_{u_0}>K$,
\begin{equation*}
\E|u_t(x)|^2=\infty\quad\text{for all}\quad t\geq t_0 \quad\text{and}\quad x\in B(0,\,R-\epsilon).
\end{equation*}
\end{theorem}

Here is a plan of the article. Section 2 contains some estimates and information needed for the proofs of the main results. Section 3 contains the proofs of Theorem \ref{theo1} and \ref{theo2}. Theorem \ref{Theo3} is proved in section 4.  In section 5, we present the proof of the remaining results.  We end this introduction with a few words about notation. We will denote the ball of radius $R$ by $B(0,\,R)$. For $x\in \R^d$, $|x|$ will be the magnitude of $x$. The letter $c$ with or without subscripts will denote a constant whose value is not relevent.

\section{Estimates}
In this section, we collect some results needed for the proof of our main results. We begin with the heat kernel of the stable process.
  \begin{itemize}
\item \begin{equation*}
p_{st}(x)=s^{-d/\alpha}p_t(s^{-1/\alpha}x).
\end{equation*}
\item $p_t(x)\geq p_t(y)$whenever $|x|\leq |y|$.
 \item For $t$ large enough so that $p_t(0)\leq 1$ and $\tau\geq 2$, we have
 \begin{equation}\label{mono}
p_t(\frac{1}{\tau}(x-y))\geq p_t(x)p_t(y).
\end{equation}
\end{itemize}
We provide a quick proof of the last inequality.  Suppose that $t$ is large enough so that $p(t,\,0)\leq 1$. Now, we have that $\frac{|x-y|}{\tau}\leq \frac{2|x|}{\tau}\vee \frac{2|y|}{\tau}\leq |x|\vee|y|.$ Therefore by the monotonicity property of the heat kernel and the fact that time is large enough, we have
\begin{equation*}
\begin{aligned}
p(t,\,\frac{1}{\tau}(x-y))&\geq p(t,\,|x|\vee|y|)\\
&\geq p(t,\,|x|)\wedge p(t,\,|y|)\\
&\geq p(t,\,|x|)p(t,\,|y|).
\end{aligned}
\end{equation*}
We will also need the following heat kernel estimates,
 \begin{equation}\label{heat}
 c_1\bigg(t^{-d/\alpha}\wedge \frac{t}{|x|^{d+\alpha}}\bigg)\leq p_t(x)\leq c_2\bigg(t^{-d/\alpha}\wedge \frac{t}{|x|^{d+\alpha}}\bigg).
 \end{equation}
Finally, we will need the following property; see \cite{Bogdan}.  Let $p_{D,t}(x,y)$ denote the heat kernel of the process killed upon exiting the ball $B(0,\,R)$.
\begin{itemize}
\item Fix $\epsilon >0$, then for all $t\leq \epsilon^\alpha$ and $x,y \in B(0,\,R-\epsilon)$, we have
\begin{equation}\label{comp}
p_{D,t}(x,y)\geq c_1 p_t(x-y)\quad\text{whenever}\quad x,\,y\in B(0,\,R-\epsilon).
\end{equation}
\end{itemize}

\begin{proposition}\label{determ}
Suppose that Assumption \ref{initial} holds. Then there exists a positive number $t_0$ such that for $t\in (0,\,t_0]$, we have
\begin{equation*}
(\cG u)_{t+t_0}(x)\geq c_1K_{u_0}\quad \text{whenever}\quad x\in B(0,\,1).
\end{equation*}
\end{proposition}
\begin{proof}
We first choose $t_0$ large enough so that $p_{t_0}(0)<1$. Using \eqref{mono}, we can write
\begin{align*}
p_{t_0}(x-y)&=p_{t_0}\left(\frac{1}{2}(2x-2y)\right)\\
&\geq p_{t_0}(2x)p_{t_0}(2y)\\
&\geq 2^{-d}p_{\frac{t_0}{2^\alpha}}(x)p_{t_0}(2y).
\end{align*}
We therefore have
\begin{align*}
(\cG u)_{t_0}(x)\geq 2^{-d}p_{\frac{t_0}{2^\alpha}}(x)\int_{\R^d}p_{t_0}(2y)u_0(y)\,\d y,
\end{align*}
which after using scaling and Assumption \ref{initial} gives
\begin{align*}
(\cG u)_{t_0}(x)&\geq 2^{-d}t_0^{-d/\alpha}p_{\frac{t_0}{2^\alpha}}(x)\int_{\R^d}p_1(2yt_0^{-1/\alpha})u_0(y)\,\d y\\
&\geq c_2K_{u_0}t_0^{-d/\alpha}p_{\frac{t_0}{2^\alpha}}(x).
\end{align*}
We now use the semigroup property to obtain
\begin{align*}
(\cG u)_{t_0+t}(x)\geq c_2K_{u_0}t_0^{-d/\alpha}p_{\frac{t_0}{2^\alpha}+t}(x).
\end{align*}
Using \eqref{heat}, we obtain
\begin{align*}
(\cG u)_{t_0+t}(x)\geq c_4K_{u_0},
\end{align*}
where $c_4$ depends on $t_0$.
\end{proof}
We will need the following only for the proof of our last theorem. Set
\begin{equation*}
(\cG_Du)_t(x):=\int_{B(0,\,R)}p_{D,t}(x,y)u_0(y)\,\d y.
\end{equation*}
\begin{proposition}\label{determ-dirich}
Let $R>0$ and suppose that Assumption \ref{initial-dirichlet} holds. Then for all $x\in B(0,\,R/2)$, we have
\begin{align*}
(\cG_Du)_t(x)\geq c_1\quad \text{whenever}\quad t\leq \left(\frac{R}{2}\right)^\alpha,
\end{align*}
where $c_1$ is some positive constant.
\end{proposition}
\begin{proof}
The proof is quite straightforward. We use Assumption \ref{initial-dirichlet} and \eqref{comp} to see that for $t\leq \left(\frac{R}{2}\right)^\alpha$, we have
\begin{align*}
(\cG_Du)_t(x)&=\int_{B(0,\,R)}p_{D,t}(x,y)u_0(y)\,\d y\\
&\geq \int_{B(0,\,R/2)}p_{D,t}(x,y)u_0(y)\,\d y\\
&\geq \int_{B(0,\,R/2)\cap\{y\in B(0,\,R/2); |x-y|\leq t^{1/\alpha} \}}p_t(x-y)u_0(y)\,\d y\\
&\geq c_2\tilde{\kappa}.
\end{align*}
\end{proof}

\begin{proposition}\label{prop-kernel}
Let $R>0$. Suppose that $t\leq \left(\frac{R}{2}\right)^\alpha$. Then for all $x_1,\,x_2 \in B(0,\,R)$, we have
\begin{equation*}
\int_{B(0,\,R)\times B(0,\,R)}p_{t-s}(x_1-y_1)p_{t-s}(x_2-y_2)f(y_1,\,y_2)\,\d y_1\d y_2\geq c_1K_{f},
\end{equation*}
where $s\leq t$ and $c_1$ is some positive constant.
\end{proposition}
\begin{proof}
Assumption \ref{color} gives
\begin{align*}
\int_{B(0,\,R)\times B(0,\,R)}&p_{t-s}(x_1-y_1)p_{t-s}(x_2-y_2)f(y_1,\,y_2)\,\d y_1\d y_2\\
&\geq K_f\int_{B(0,\,R)\times B(0,\,R)}p_{t-s}(x_1-y_1)p_{t-s}(x_2-y_2)\,\d y_1\d y_2.
\end{align*}
We now use the fact $t\leq \left(\frac{R}{2}\right)^\alpha$ to observe that if for $i=1,\,2$, we set
\begin{equation*}
\mathcal{A}_i:=\{y_1\in B(0,\,R); |x_i-y_i|\leq (t-s)^{1/\alpha} \},
\end{equation*}
then $|\mathcal{A}_i|=c_2|t-s|^{d/\alpha}$ for some $c_2$. We therefore have
\begin{align*}
\int_{B(0,\,R)\times B(0,\,R)}&p_{t-s}(x_1-y_1)p_{t-s}(x_2-y_2)\,\d y_1\d y_2\\
&\geq \int_{\mathcal{A}_1\times \mathcal{A}_2}p_{t-s}(x_1-y_1)p_{t-s}(x_2-y_2)\,\d y_1\d y_2\\
&\geq c_3,
\end{align*}
where to obtain the last inequality, we have used the heat kernel estimates given by \eqref{heat}.  Combining the above, we have the required inequality.
\end{proof}
We now present the renewal inequalities needed to conclude to non-existence.
\begin{proposition}\label{volterra}
Fix $T>0$ and suppose that $g$ is a non-negative function satisfying the following non-linear integral inequality,
\begin{align*}
g(t)\geq A+B\int_0^t\frac{g(s)^{1+\gamma}}{(t-s)^{1/\alpha}}\,\d s,\quad \text{for}\quad 0<t\leq T,
\end{align*}
where $A$, $B$ and $\gamma$ are positive numbers. Then for any $t_0\in (0,\,T]$, there exists an $A_0$ such that for $A>A_0$
\begin{align*}
g(t)=\infty\quad\text{whenever}\quad{t\geq t_0}.
\end{align*}
\end{proposition}
\begin{proof}
From the integral inequality, we have
\begin{align*}
g(t)\geq A+\frac{B}{T^{1/\alpha}}\int_0^tg(s)^{1+\gamma}\,\d s,\quad \text{for}\quad t\leq T.
\end{align*}
By the comparison principle, it is enough to consider
\begin{align*}
g(t)=A+\frac{B}{T^{1/\alpha}}\int_0^tg(s)^{1+\gamma}\,\d s,\quad \text{for}\quad t\leq T,
\end{align*}
which amounts to looking at the following non-linear ordinary differential equation,
\begin{equation*}
\frac{g'(t)}{g(t)^{1+\gamma}}=\frac{B}{T^{1/\alpha}},
\end{equation*}
with initial condition $A$. The solution of this equation is given by
\begin{equation*}
\frac{1}{g(t)^{\gamma}}=\frac{1}{A^\gamma}-\frac{\gamma Bt}{T^{1/\alpha}} \quad\text{for}\quad t\leq T.
\end{equation*}
Hence by blowup occur at $t=\frac{T^{1/\alpha}}{A^\gamma B\gamma}$. Fix any $t_0<T$ and take $A>\left(\frac{T^{1/\alpha}}{B\gamma t_0}\right)^{1/\gamma}$, we obtain blow-up before time $t_0$ and since $g(t)$ is increasing on $(0,\infty)$, we have the required result.
\end{proof}
We have the following result which is slightly different from the previous one.
\begin{proposition}\label{rem-volterra}
Suppose that $0<(1+\gamma)/\alpha<1$. If $f$ is a non-negative function satisfying the following non-linear integral inequality,
\begin{align*}
g(t)\geq A+B\int_0^t\frac{g(s)^{1+\gamma}}{(t-s)^{1/\alpha}}\,\d s,\quad \text{for}\quad t>0
\end{align*}
where $A$, $B$ and $\gamma$ are positive numbers, then for any $A>0$ there exists $t_0>0$ such that $g(t)=\infty$ for all $t\geq t_0$.
\end{proposition}
\begin{proof}
We start with the following consequence of the the inequality.
\begin{align*}
g(t)\geq A+B\int_0^t\frac{g(s)^{1+\gamma}}{t^{1/\alpha}}\,\d s,\quad \text{for}\quad t>0.
\end{align*}
We can always assume that $t_0>1$. So after setting $h(t):=t^{1/\alpha}g(t)$, the above reduces to
\begin{align*}
h(t)\geq At^{1/\alpha}+B\int_0^t\frac{h(s)^{1+\gamma}}{s^{(1+\gamma)/\alpha}}\,\d s,\quad \text{for}\quad t>0.
\end{align*}
We will take $t\geq1$,
then we get 
\begin{align*}
h(t)\geq A+B\int_1^t\frac{h(s)^{1+\gamma}}{s^{(1+\gamma)/\alpha}}\,\d s,\quad \text{for}\quad t\geq1.
\end{align*}
 By the comparison principle, we need to look at the following ordinary differential equation,
\begin{equation*}
\frac{h'(t)}{[h(t)]^{1+\gamma}}=\frac{B}{t^{(1+\gamma)/\alpha}}, \quad\text{for}\quad t\geq 1
\end{equation*}
with initial condition $h(1)=A$. This can be explicitly solved to conclude that because $(1+\gamma)/\alpha<1$, we have $h(t)=\infty$ for all $t\geq t_0$.
\end{proof}
\begin{remark}The same conclusion holds for $(1+\gamma)/\alpha\geq1$. This is because we can always write $\gamma=\gamma_0+(\gamma-\gamma_0)$ so that $\gamma_0<\gamma$ and $(1+\gamma_0)/\alpha<1$. We now use the fact that $g(t)>A$ for all $t>0$ and use the integral inequality to come up with
\begin{align*}
g(t)\geq A+BA^{\gamma-\gamma_0}\int_0^t\frac{g(s)^{1+\gamma_0}}{(t-s)^{1/\alpha}}\,\d s,\quad \text{for}\quad t>0.
\end{align*}
The conclusion now easily follows from this and Proposition \ref{rem-volterra}.
\end{remark}

\section{Proof of Theorems \ref{theo1} and \ref{theo2}}
\begin{proof}[\bf Proof of Theorem \ref{theo1}.]
We begin with the mild formulation of the solution i.e \eqref{mild}, take second moment and use the Walsh isometry to obtain
\begin{align*}
\E|u_t(x)|^2&=|(\cG u)_t(x)|^2+\int_0^t\int_\R p^2_{t-s}(x-y)\E|\sigma(u_s(y))|^2\,\d y\,\d s\\
&=I_1+I_2.
\end{align*}
The fact that the initial condition is bounded below gives
\begin{align*}
I_1\geq \kappa^2.
\end{align*}
A little more work allows to bound $I_2$ as follows
\begin{align*}
I_2&\geq \int_0^t\left(\inf_{x\in \R}\E|u_t(x)|^2\right)^{\1+\gamma}\int_\R p^2_{t-s}(x-y)\,\d y\,\d s\\
&\geq c_2\int_0^t\left(\inf_{x\in \R}\E|u_s(x)|^2\right)^{\1+\gamma}\frac{1}{(t-s)^{1/\alpha}}\d s.
\end{align*}
Note that to obtain the inequalities, we have used the growth condition on $\sigma$ and the heat kernel estimates \eqref{heat}.
Upon setting
\begin{align*}
F(s):=\inf_{x\in \R}\E|u_s(x)|^2
\end{align*}
and combining the above inequalities, we obtain
\begin{equation*}
F(t)\geq c_3+c_4\int_0^t\frac{F(s)^{1+\gamma}}{(t-s)^{1/\alpha}}\,\d s.
\end{equation*}
The result  now follows from the argument of Proposition \ref{rem-volterra}.
\end{proof}

The proof of Theorem \ref{theo2} will require a new idea. Instead of the looking for a non-linear renewal inequality which involves the second moment, we look at one which involve a different quantity. The proof of the following proposition makes this more precise.
\begin{proposition}\label{prop-colored}
There exists a $t_0>0$ and $\kappa_0>0$ such that whenever the lower bound of $u_0$ in \eqref{minumum-values-u0} satisfies $\kappa>\kappa_0$, then  for all $x,\,y\in B(0,\,1)$,
\begin{align*}
\E|u_t(x)u_t(y)|=\infty\quad\text{for all}\quad t\geq t_0.
\end{align*}
\end{proposition}
\begin{proof}
We start off with \eqref{mild} to obtain
\begin{align*}
\E|&u_t(x)u_t(y)|\\
&\geq \cG u_t(x)\cG u_t(y)+\int_0^t\int_{\R^d\times\R^d}p_{t-s}(x-z)p_{t-s}(y-w)f(z,w)\left(\E|u_s(z)u_s(w)|\right)^{1+\gamma}\,\d z\d w\d s\\
&:=I_1+I_2.
\end{align*}
We look at the term $I_1$ first. The fact that the initial condition is bounded below by $\kappa$ gives
\begin{equation*}
I_1\geq \kappa^2.
\end{equation*}
We now assume that $t<\left(\frac{1}{2}\right)^\alpha$ and use Proposition \ref{prop-kernel} to obtain
\begin{align*}
I_2&\geq \int_0^t\left(\inf_{z,\,w\in B(0,\,1)}\E|u_s(z)u_s(w)|\right)^{1+\gamma}\int_{B(0,\,1)\times B(0,\,1)}p_{t-s}(x-z)p_{t-s}(y-w)f(z,w)\,\d z\d w\d s\\
&\geq c_1K_f\int_0^t\left(\inf_{z,\,w\in B(0,\,1)}\E|u_s(z)u_s(w)|\right)^{1+\gamma} \d s.
\end{align*}
We now set
\begin{equation*}
G(s):=\inf_{x,\,y\in B(0,\,1)}\E|u_s(x)u_s(y)|.
\end{equation*}
We combine the estimates above to obtain
\begin{align*}
G(t)\geq \kappa^2+K_f\int_0^tG(s)^{1+\gamma}\,\d s \quad\text{for}\quad t\leq\left(\frac{1}{2}\right)^{\alpha}.
\end{align*}
The proof of Proposition \ref{volterra} implies that by taking $\kappa$ big enough, we can make sure that $t_0$ is as small as we wish. This completes the proof of the result.
\end{proof}

\begin{proof}[\bf Proof of Theorem \ref{theo2}]
With the above proposition at our disposal, we can readily prove the theorem. Indeed, from the mild formulation, we have
\begin{align*}
\E|&u_t(x)|^2\\
&\geq \kappa^2+\int_{t_0}^t\int_{B(0,\,1)\times B(0,\,1)}p_{t-s}(x-y)p_{t-s}(x-w)f(y,w)\left(\E||u_s(y)u_s(w)|\right)^{1+\gamma}\d y\d w\d s.
\end{align*}
The result now follows from the fact that all the functions involved on the right of the last inequality are strictly positive.
\end{proof}

\section{Proof of Theorem \ref{Theo3}}
In this section, we will always be assuming that the correlation function is given by the Riesz kernel, that is
\begin{align*}
f(z)=\frac{1}{|z|^\beta},\ \ \ \beta< d\wedge \alpha.
\end{align*}
$u_t$ will be the solution to \eqref{eq:colored}.
\begin{proposition}
There exists a constant $c_1$ such that
\begin{align*}
\int_{\R^d\times\R^d}p_t(x-z)p_t(y-w)f(z-w)\,\d z\d w\geq \frac{c_1}{t^{\beta/\alpha}}\quad\text{whenever}\quad x,\,y\in B(0,\,t^{1/\alpha}).
\end{align*}
\end{proposition}
\begin{proof}
The proof is quite straightforward. We use the bounds given by \eqref{heat} to write
\begin{align*}
\int_{\R^d\times\R^d}p_t(x-z)&p_t(y-w)f(z-w)\,\d z\d w\\
&\geq \int_{B(0,\,t^{1/\alpha})\times B(0,\,t^{1/\alpha})}p_t(x-z)p_t(y-w)f(z-w)\,\d z\d w\\
&\geq \frac{c_1}{t^{\beta/\alpha}}.
\end{align*}
\end{proof}
The following result is now straightforward.
\begin{proposition}\label{correlation-lower-bound}
Fix $t>0$. For $x,\,y\in B(0,\,t^{1/\alpha})$, we have
\begin{align*}
\E|u_t(x)u_t(y)|\geq c_1t^{(\alpha-\beta)/\alpha},
\end{align*}
where $c_1$ is some constant.
\end{proposition}
\begin{proof}
We obviously have $\E|u_t(x)u_t(y)|\geq \kappa$.  An application of the above proposition gives
\begin{align*}
\E|u_t(x)u_t(y)|&\geq \kappa^{1+\lambda}\int_0^t\int_{\R^d\times\R^d}p_s(x-z)p_s(y-w)f(z-w)\,\d z\d w\d s\\
&\geq c_1t^{(\alpha-\beta)/\alpha}.
\end{align*}
We have also used the fact that initial condition is non-negative.
\end{proof}

We now have this technical result which is the final hurdle for the proof of Theorem \ref{Theo3}.

\begin{proposition}
Fix $t>0$ and let $t_0\leq t/6$. Then for $x,\,y\in B(0,\,t^{1/\alpha})$, we have
\begin{align*}
\int_0^t\int_{\R^d\times\R^d}p_{t+t_0-s}(x-z)p_{t+t_0-s}(y-w)\E|u_s(z)u_s(w)|^{1+\gamma}f(z-w)\,\d z\,\d w\geq c_1t^{2(\alpha-\beta)/\alpha},
\end{align*}
for some constant $c_1$.
\end{proposition}
\begin{proof}
Using the fact that $\E|u_t(x)u_t(y)|\geq \kappa$ and  Proposition \ref{correlation-lower-bound}, we write
\begin{align*}
\int_0^t&\int_{\R^d\times\R^d}p_{t+t_0-s}(x-z)p_{t+t_0-s}(y-w)\E|u_s(z)u_s(w)|^{1+\gamma}f(z-w)\,\d z\,\d w\,\d s\\
&\geq \kappa^{\gamma}\int_0^ts^{(\alpha-2\beta)/\alpha}\int_{B(0,\,s^{1/\alpha})\times B(0,\,s^{1/\alpha})}p_{t+t_0-s}(x-z)p_{t+t_0-s}(y-w)\,\d z\,\d w \,\d s\\
&\geq \kappa^{\gamma}\int_{(t+t_0)/2}^{3(t+t_0)/4}s^{(\alpha-2\beta)/\alpha}\int_{B(0,\,(t+t_0-s)^{1/\alpha})\times B(0,\,(t+t_0-s)^{1/\alpha})}p_{t+t_0-s}(x-z)p_{t+t_0-s}(y-w)\,\d z\,\d w \,\d s.
\end{align*}
We have used the fact that if $s\geq \frac{t+t_0}{2}$, then $s\geq t-s+t_0$. We now observe that since $s\leq 3(t+t_0)/4$, we have $s\leq 3(t+t_0-s)$. We can now see that
\begin{align*}
|x-z|&\leq t^{1/\alpha}+(t-s+t_0)^{1/\alpha}\\
&\leq (t-s+t_0+s)^{1/\alpha}+(t-s+t_0)^{1/\alpha}\\
&\leq c_2(t-s+t_0)^{1/\alpha},
\end{align*}
for some constant $c_2$.  We now use the bound on the heat kernel to conclude that
\begin{align*}
\int_{(t+t_0)/2}^{3(t+t_0)/4}s^{(\alpha-2\beta)/\alpha}&\int_{B(0,\,(t+t_0-s)^{1/\alpha})\times B(0,\,(t+t_0-s)^{1/\alpha})}p_{t+t_0-s}(x-z)p_{t+t_0-s}(y-w)\,\d z\,\d w \,\d s\\
&\geq c_3t^{2(\alpha-\beta)/\alpha}.
\end{align*}
\end{proof}
As the proof shows, the above is not sharp. But this is sufficient for our needs here.
\begin{proof}[Proof of Theorem \ref{Theo3}]
The mild formulation and the fact that initial condition is bounded below gives
\begin{align}\label{time shift}
\E |u_{T+t}(x)&u_{T+t}(y)|\nonumber\\
&\geq \kappa^2+\int_0^{T+t}\int_{\R^d\times \R^d}p_{T+t-s}(x-z)p_{T+t-s}(y-w)\E|u_s(z)u_s(w)|^{1+\lambda}f(z-w)\d z\d w\d s\nonumber\\
&= \kappa^2+\int_0^{T}\int_{\R^d\times \R^d}p_{T+t-s}(x-z)p_{T+t-s}(y-w)\E|u_s(z)u_s(w)|^{1+\gamma}f(z-w)\d z\d w\d s\nonumber\\
&+\int_0^{t}\int_{\R^d\times \R^d}p_{t-s}(x-z)p_{t-s}(y-w)\E|u_{T+s}(z)u_{T+s}(w)|^{1+\gamma}f(z-w)\d z\d w\d s.
\end{align}
The second term appearing in the above is obtained via a change of variable. We will use this `time shift' trick in other proofs as well. We will assume that $T\gg 1$ and $t\leq T/6$, so that we can use the previous Proposition to bound the second term. To bound  the third term, we use similar arguments as in the proof of Theorem \ref{theo2}. So upon setting
\begin{align*}
F(s):=\inf_{x,\,y\in B(0,\,1)}\E |u_{T+s}(x)u_{T+s}(y)|,
\end{align*}
we obtain
\begin{align*}
F(t)\geq \kappa^2+c_1T^{2(\alpha-\beta)/\alpha}+c_2\int_0^{t}F(s)^{1+\gamma}\,\d s.
\end{align*}
So as long as $\kappa$ is strictly positive, we will have blow up of $F$ for any fixed small time; we just need to take $T$ large enough. The result now follows from the proof of Theorem \ref{theo2}.
\end{proof}

\section{Proof of Theorems \ref{energy-white}, \ref{energy-colored} and \ref{Dirichlet}}
The proof of Theorem \ref{energy-white} relies on the following proposition.
\begin{proposition}
Suppose that Assumption \ref{initial} holds. Then, there exists a $\tilde{t},\,K>0$ such that for all $t\geq \tilde{t}$, we have
\begin{equation*}
\inf_{x\in B(0,\,1)}\E|u_t(x)|^2=\infty,
\end{equation*}
whenever $K_{u_0}>K$.
\end{proposition}
\begin{proof}
As with the previous proof, we start off with the following using the Walsh isometry
\begin{align*}
\E|u_t(x)|^2&=|(\cG u)_t(x)|^2+\int_0^t\int_\R p_{t-s}^2(x-y)\E|\sigma(u_s(y))|^2\,\d y\,\d s.
\end{align*}
We can always assume that $\tilde{t}$ to be large. Otherwise, there is nothing to prove. So instead of looking at time $t$, we will look at $t+t_0$ and fix $t_0>0$ later. We have
\begin{align*}
\E|u_{t+t_0}(x)|^2&=|(\cG u)_{t+t_0}(x)|^2+\int_0^{t+t_0}\int_\R p_{t+t_0-s}^2(x-y)\E|\sigma(u_s(y))|^2\,\d y\,\d s.
\end{align*}
We now make an appropriate change of variables and as in \eqref{time shift}, we obtain
\begin{align*}
\E|u_{t+t_0}(x)|^2&\geq |(\cG u)_{t+t_0}(x)|^2+\int_0^t\int_\R p_{t-s}^2(x-y)\E|\sigma(u_{s+t_0}(y))|^2\,\d y\,\d s\\
&:= I_1+I_2.
\end{align*}
We will assume that $t<1$ for most of the rest of the proof. We find a lower bound  on $I_1$ first. Let $x\in B(0,\,1)$, then we fix $t_0$ as in Proposition \ref{determ}. This gives us
\begin{align*}
I_1&\geq c_1K_{u_0}^2,
\end{align*}
where the constant $c_1$ depends on $t_0$. We now look at the second term.
\begin{align*}
I_2&\geq \int_0^t\left(\inf_{y\in B(0,\,1)}\E|u_{s+t_0}(x)|^2\right)^{1+\gamma}\int_{B(0,\,1)}p_{t-s}^2(x-y)\,\d y\,\d s\\
&\geq c_2\int_0^t\left(\inf_{y\in B(0,\,1)}\E|u_{s+t_0}(x)|^2\right)^{1+\gamma}\frac{1}{(t-s)^{1/\alpha}}\,\d s
\end{align*}
Setting $H(s):=\inf_{x\in B(0,\,1)}\E|u_{s+t_0}(x)|^2$, we obtain
\begin{equation*}
H(t)\geq c_1K_{u_0}^2+c_2\int_0^t\frac{H(s)^{1+\gamma}}{(t-s)^{1/\alpha}}\,\d s\quad\text{for}\quad t\leq1.
\end{equation*}
An application of Proposition \ref{volterra} gives the desired result.
\end{proof}
\begin{proof}[\bf Proof of Theorem \ref{energy-white}]
The proof of the theorem now follows from
\begin{align*}
\E|u_t(x)|^2&\geq |(\cG u)_t(x)|^2+\int_0^t\int_\R p_{t-s}^2(x-y)\left(\E|u_s(y))|^2\right)^{1+\gamma}\,\d y\,\d s\\
&\geq |(\cG u)_t(x)|^2+\int_{\tilde{t}}^t\int_{B(0,\,1)} p_{t-s}^2(x-y)\left(\E|u_s(y))|^2\right)^{1+\gamma}\,\d y\,\d s,
\end{align*}
where $t>\tilde{t}$. Now for any $x\in \R$, since the first term of the above display is strictly positive, we have the proof of the theorem.
\end{proof}

\begin{proposition}
Let $u_t$ be the solution to \eqref{eq:colored}. Suppose that Assumptions \ref{color} and \ref{initial} hold. Then, there exists a $\tilde{t}>0$ such that for all $t\geq \tilde{t}$, we have
\begin{equation*}
\inf_{x,\,y\in B(0,\,1)}\E|u_t(x)u_t(y)|=\infty,
\end{equation*}
whenever $K_{u_0}>K$ where $K$ is some positive constant.
\end{proposition}

\begin{proof}
Our starting point will yet again be the mild formulation from which we obtain
\begin{align*}
\E&|u_t(x)u_t(y)|\\
&\geq (\cG u)_t(x)(\cG u)_t(y)+\int_0^t\int_{\R^d\times\R^d}p_{t-s}(x-z)p_{t-s}(y-w)f(z-w)\E|\sigma(u_s(z))\sigma(u_s(w))|\,\d z \d w \d s.
\end{align*}
This follows essentially the same idea as in the previous proofs. The key idea is to take $t_0$ as in Proposition \ref{determ} and set
\begin{equation*}
F(s):=\inf_{x,\,y\in B(0,\,1)}\E|u_{s+t_0}(x)u_{s+t_0}(y)|.
\end{equation*}
Using the ideas in Proposition \ref{prop-colored}, we obtain
\begin{align*}
F(t)\geq c_1K_{u_0}^2+c_2K_f\int_0^tF(s)^{1+\gamma}\d s,
\end{align*}
valid for a suitable range of $t$. We now use Proposition \ref{volterra} to finish the proof.
\end{proof}
\begin{proof}[\bf Proof of Theorem \ref{energy-colored}]
With the above Proposition, the proof of theorem is now very similar to that of Theorem \ref{energy-white} and is therefore omitted.
\end{proof}

The proof of our final theorem follow a similar pattern to the proofs of the previous results. We emphasis that in the case of \eqref{eq:dir}, the mild solution is given by 
\begin{equation}
u_t(x)=
(\cG_D u)_t(x)+ \int_{\R^d}\int_0^t p_{D, t-s}(x-y)\sigma(u_s(y))F(\d s\,\d y).
\end{equation}
\begin{proof}[\bf Proof of Theorem \ref{Dirichlet}]
As before, we have
\begin{align*}
\E|&u_t(x)u_t(y)|\\
&\geq(\cG_D u)_t(x)(\cG_D u)_t(y)\\
&+\int_0^t\int_{B(0,\,R)\times B(0,\,R)}p_{D,t-s}(x-z)p_{D,t-s}(y-w)f(z-w)\left(\E|u_s(z))u_s(w)|\right)^{1+\gamma}\,\d z\,\d w\,\d s\\
&:=I_1+I_2.
\end{align*}
We look at $I_1$ first. By Proposition \ref{determ-dirich}, if $x,\,y \in B(0,\,R/2)$ and $t$ is small enough, we have $I_1\geq c_2\kappa^2$. We now turn our attention to the second term.
\begin{align*}
I_2&\geq K_f\int_0^t\left(\inf_{x,\,y\in B(0,\,R/2)}\E|u_s(x)u_s(y)| \right)^{1+\gamma}\\
&\times \int_{B(0,\,R/2)\times B(0,\,R/2)}p_{D,t-s}(x-z)p_{D,t-s}(y-w)f(z-w)\,\d z\,\d w\,\d s.
\end{align*}
We now use  equation \eqref{comp} with $\epsilon=R/4$  and Proposition \ref{prop-kernel} with $t\leq (\frac{R}{4})^\alpha$ to obtain
\begin{align*}
\int_{B(0,\,R/2)\times B(0,\,R/2)}p_{D,t-s}&(x-z)p_{D,t-s}(y-w)f(z-w)\,\d z\,\d w\\
&\geq \int_{B(0,\,R/2)\times B(0,\,R/2)}p_{t-s}(x-z)p_{t-s}(y-w)f(z-w)\,\d z\,\d w\\
&\geq K_f.
\end{align*}
This gives
\begin{align*}
I_2&\geq c_3K_f\int_0^t\left(\inf_{x,\,y\in B(0,\,R/2)}\E|u_s(x)u_s(y)| \right)^{1+\gamma}\,\d s.
\end{align*}
By combining the above inequalities and setting
\begin{equation*}
F(s):=\inf_{x,\,y\in B(0,\,R/2)}\E|u_s(x)u_s(y)|,
\end{equation*}
we have
\begin{align*}
F(s)\geq c_2\kappa^2+c_3K_f\int_0^t F(s)^{1+\gamma}\,\d s.
\end{align*}
Using ideas in the proof of Proposition \ref{volterra}, we see that for any $t_0<\left(\frac{R}{2}\right)^\alpha$, there exists a  $\kappa_0$ such that for $\kappa>\kappa_0$, $F(s)=\infty$ for all $s\geq t_0$. To finish the proof, we make the following observation.
\begin{align*}
\E|&u_t(x)|^2\\
&\geq |(\cG u)_t(x)|^2\\
&+\int_{t_0}^t\int_{\R^d\times\R^d}p_{D,t-s}(x-z)p_{D,t-s}(x-w)f(z-w)\left(\E|(u_s(z))(u_s(y))|\right)^{1+\gamma}\,\d z\,\d w\,\d s.
\end{align*}
By the positivity of all the relevant terms involved, we obtain the result.
\end{proof}

\bibliography{Foon}
\end{document}